\documentclass{amsart}

\newtheorem{theorem}[equation]{Theorem}
\newtheorem{lemma}[equation]{Lemma}
\newtheorem{corollary}[equation]{Corollary}
\newtheorem{proposition}[equation]{Proposition}

\numberwithin{equation}{section}

\usepackage{amscd,amsmath,amssymb,verbatim}

\begin{document}

\title{Exponential sums nondegenerate \linebreak relative to a lattice} 
\author{Alan Adolphson}
\address{Department of Mathematics\\
Oklahoma State University\\
Stillwater, Oklahoma 74078}
\email{adolphs@math.okstate.edu}
\author{Steven Sperber}
\address{School of Mathematics\\
University of Minnesota\\
Minneapolis, Minnesota 55455}
\email{sperber@math.umn.edu}
\begin{abstract}
Our previous theorems on exponential sums often did not apply or did not give
sharp results when certain powers of a variable appearing in the polynomial
were divisible by $p$.  We remedy that defect in this paper by
systematically applying ``$p$-power reduction,'' making it possible to 
strengthen and extend our earlier results. 
\end{abstract}
\maketitle

\section{Introduction}

In the papers \cite{AS1}--\cite{AS4} we established properties of the
$L$-functions of exponential sums on affine space ${\mathbb A}^n$ and the
torus ${\mathbb T}^n$.  The purpose of this article is to prove a general
result that leads to a sharpening of the theorems of those papers.

Let $p$ be a prime, let $q=p^r$, and let ${\mathbb F}_q$ be the field of $q$
elements.  Let $f\in{\mathbb F}_q[x_1^{\pm 1},\dots,x_n^{\pm 1}]$ be a Laurent 
polynomial, say,
\begin{equation}
f = \sum_{j\in J} a_jx^j,
\end{equation}
where $a_j\in{\mathbb F}_q^\times$ and $J$ is a finite subset of ${\mathbb
  Z}^n$.  Let ${\mathbb Z}\langle J\rangle$ be the subgroup of ${\mathbb Z}^n$
generated by the elements of $J$.  By the basic theory of abelian groups,
there exists a basis ${\bf u}_1,\dots,{\bf u}_n$ for ${\mathbb Z}^n$ and
integers $d_1,\dots,d_k$ such that $d_1{\bf u}_1,\dots,d_k{\bf u}_k$ is a
basis for ${\mathbb Z}\langle J\rangle$.  After a coordinate change on
${\mathbb T}^n$, we may assume that ${\bf u}_1,\dots,{\bf u}_n$ is the
standard basis.  The Laurent polynomial $f$ may then be written in the form
\[ f = g(x_1^{d_1},\dots,x_k^{d_k}) \]
for some $g\in{\mathbb F}_q[x_1^{\pm 1},\dots,x_k^{\pm 1}]$.  Write $d_i =
p^{b_i}e_i$ for each $i$, where $b_i\geq 0$ and $(e_i,p)=1$.  Since raising
to the $p$-th power is an automorphism of ${\mathbb F}_q$, the exponential
sums associated to the polynomials $f$ and $g(x_1^{e_1},\dots,x_k^{e_k})$ are
identical.  Furthermore, the theorems of \cite{AS1}--\cite{AS4} generally
produce sharper results when applied to $g(x_1^{e_1},\dots,x_k^{e_k})$ than
when applied to $f$.  (Thus there is no improvement over our earlier
work if $p\nmid[{\mathbb Z}^k:{\mathbb Z}\langle J\rangle]$.)  We refer to 
$g(x_1^{e_1},\dots,x_k^{e_k})$ as the {\it $p$-power reduction\/} of $f$.

Over ${\mathbb A}^n$, the technique of $p$-power reduction is less
versatile because one cannot make the same sorts of coordinate
changes.  One has a standard toric decomposition of ${\mathbb A}^n$,
${\mathbb A}^n = \bigcup_{A\subseteq\{1,\dots,n\}}{\mathbb T}_A$, 
where ${\mathbb T}_A$ denotes the $|A|$-dimensional torus with
coordinates $\{x_i\}_{i\in A}$.  Given $f\in{\mathbb
  F}_q[x_1,\dots,x_n]$, one can try to analyze the corresponding
exponential sum on ${\mathbb A}^n$ by analyzing its restriction to
each of these tori, but the picture is complicated by the fact that
$p$-power reduction may require different coordinate changes on
different tori.  It thus seems worthwhile to generalize our previous
results to apply directly to the polynomial as given, to avoid the
task of performing $p$-power reduction on a case-by-case basis.

Let $M_J$ be the {\it prime-to-$p$ saturation\/} of ${\mathbb Z}\langle
J\rangle$, 
\[ M_J = \{u\in{\mathbb Z}^n\mid \text{$ku\in{\mathbb Z}\langle J\rangle$ for
  some $k\in{\mathbb Z}$ satisfying $(k,p) = 1$}\}, \]
and let ${\mathbb R}\langle J\rangle$ denote the subspace of ${\mathbb R}^n$
spanned by the elements of $J$.  We will get a strengthening of our earlier
results when $M_J$ is a proper subset of ${\mathbb
  Z}^n\cap{\mathbb R}\langle J\rangle$.  Let  
\[ [{\mathbb Z}^n\cap{\mathbb R}\langle J\rangle:{\mathbb Z}\langle J\rangle]
= p^ae, \] 
where $a\geq 0$ and $(e,p)=1$.  Then 
\begin{equation}
[{\mathbb Z}^n\cap{\mathbb R}\langle
J\rangle: M_J]=p^a, 
\end{equation}
so $M_J\neq {\mathbb Z}^n\cap{\mathbb R}\langle J\rangle$ if and only if
$a>0$.

Part of the motivation for this work was a desire to understand a theorem of 
Katz\cite[Theorem~3.6.5]{K} from our point of view.  Suppose that
$f\in{\mathbb F}_q[x_1,\dots,x_n]$ is a homogeneous polynomial of degree
$d=p^ke$, $(e,p)=1$, $k>1$.  Katz showed that if $f=0$ defines a smooth
hypersurface in ${\mathbb P}^{n-1}$, then the $L$-function associated to the
exponential sum defined by $f$ (see Section~2 for the definition) is a
polynomial ($n$ odd) or the reciprocal of a polynomial ($n$ even) of degree 
\[ \frac{1}{p^k}((d-1)^n + (-1)^n(p^k-1)) \]
all of whose reciprocal roots have absolute value $q^{n/2}$.
Note that in this situation $[{\mathbb Z}^n:M_J]=p^k$.  Our earlier
results (\cite{AS3}) do not apply to polynomials of degree divisible by $p$.
However, we show here that when $M_J$ is a proper subset of ${\mathbb Z}^n$
one can weaken the definition of nondegeneracy used in \cite{AS3} and still
deduce conclusions analogous to those of that article.  In particular, we show
that the above theorem of Katz is true as well for non-homogeneous
polynomials, provided that the homogeneous part of highest degree defines a
smooth hypersurface in ${\mathbb P}^{n-1}$ and $[{\mathbb
  Z}^n:M_J]=p^k$.  In other words, the conclusion remains true when one
perturbs the smooth homogeneous polynomial by adding arbitrary terms of
degrees $p^ke'$, $e'<e$. 

This generalization of Katz's theorem (Proposition 5.1 below) will be
derived as a consequence of Theorem~4.20.  Another consequence of that
theorem is the following result.  Consider the Dwork family of
hypersurfaces  
\[ x_1^n + \dots + x_n^n + \lambda x_1\dots x_n = 0 \]
in ${\mathbb P}^{n-1}$.  If $n=p^ke$, where $k\geq 1$ and $(p,e) = 1$,
and $\lambda\neq 0$, this hypersurface is singular (except for
$n=2,3$).  We show (Corollary 5.9 below) that the zeta function of
this hypersurface has the form 
\[ Z(t) = \frac{R(t)^{(-1)^{n-1}}}{(1-t)(1-qt)\dots (1-q^{n-2}t)}, \]
where $R(t)$ is a polynomial of degree
\[ (p^k-1)e^{n-1} + e^{-1}((e-1)^n + (-1)^n(e-1)) \]
all of whose reciprocal roots have absolute value $q^{(n-2)/2}$.

As another example, we strengthen the classical theorem of Chevalley-Warning.
Let $f=\sum_{j\in J}a_jx^j\in{\mathbb F}_q[x_1,\dots,x_n]$ and let $N(f)$
denote the number of solutions of $f=0$ with coordinates in ${\mathbb F}_q$.
Let ${\mathbb N}$ denote the nonnegative integers, let ${\mathbb N}_+$ denote
the positive integers, and let $J' = \{(j,1)\in{\mathbb N}^{n+1}\mid j\in
J\}$.  Let $\Delta(J')$ denote the convex hull of $J'\cup\{(0,\dots,0)\}$ in
${\mathbb R}^{n+1}$.  
\begin{theorem}
Let $\mu$ be the smallest positive integer such that $\mu\Delta(J')$, the
dilation of $\Delta(J')$ by the factor $\mu$, contains a point of
$M_{J'}\cap({\mathbb N}_+)^{n+1}$.  Then ${\rm ord}_q\,N(f)\geq \mu-1$, where
${\rm ord}_q$ denotes the $p$-adic valuation normalized by ${\rm ord}_q\,q =
1$.
\end{theorem}

For example, the equation $\sum_{i=1}^n x_i^{p^{k_i}}=0$ has $q^{n-1}$
solutions: since raising to the $p$-th power is an automorphism of
${\mathbb F}_q$, one can assign arbitrary values to
$x_1,\dots,x_{n-1}$ and there will be a unique value of $x_n$
satisfying the equation.  One checks that $M_{J'} = {\mathbb Z}\langle
J'\rangle$ is the lattice generated by the
$\{(0,\dots,0,p^{k_i},0,\dots,0,1)\}_{i=1}^n$, so $\mu=n$ and
Theorem~1.3 gives the precise divisibility by $q$.  

For a more subtle example, let $p=3$, $n=3$, and consider the polynomial
\[ f = x_1x_2^2 + x_2x_3^2 + x_1^2x_3. \]
The lattice $M_{J'} = {\mathbb Z}\langle J'\rangle$ is the rank-three
sublattice of ${\mathbb Z}^4$ with basis the vectors ${\bf u}_1 = (1,2,0,1)$,
${\bf u}_2 = (0,1,2,1)$, ${\bf u}_3 = (2,0,1,1)$.  The only point of
$\Delta(J')\cap({\mathbb N}_+)^4$ is $(1,1,1,1)$ and one has
\begin{equation}
(1,1,1,1) = \frac{1}{3}({\bf u}_1 + {\bf u}_2 + {\bf u}_3),
\end{equation}
thus $(1,1,1,1)\not\in M_{J'}$.  It follows that $\mu>1$, so Theorem~1.3
implies that $N(f)$ is divisible by $3^r$. (In fact, ${\bf u}_1 + {\bf u}_2\in
M_{J'}\cap ({\mathbb N}_+)^4$, so $\mu = 2$.) On the other hand, since the
degree of $f$ equals the number of variables, the classical Chevalley-Warning
Theorem does not predict the divisibility of $N(f)$ by $3$.  If we take the
same polynomial $f$ but assume $p\neq 3$, then (1.4) shows that
$(1,1,1,1)\in M_{J'}$, so $\mu=1$ and Theorem~1.3 does not predict any
divisibility by $p$.

The first main idea of this paper is that when computing the action of
Dwork's Frobenius operator, which gives the $L$-function of the
exponential sum on the torus, one can ignore the action of Frobenius
on power series whose exponents lie outside of $M_J$ since such power
series contribute nothing to the spectral theory of Frobenius.  This
idea is explained in Section 2.  The second main idea is the notion of
nondegeneracy relative to a lattice, which is introduced in
Section~4.  It guarantees that the $p$-power reduction of $f$ will be
nicely behaved.  This leads to precise formulas for the degree of the
$L$-function and the number of roots of a given archimedian weight.

\section{Trace Formula}

Let $\Psi:{\mathbb F}_q\to{\mathbb Q}(\zeta_p)$ be a nontrivial additive
character and define
\[ S_m({\mathbb T}^n,f) = \sum_{x\in{\mathbb T}^n({\mathbb F}_{q^m})}
\Psi({\rm Tr}_{{\mathbb F}_{q^m}/{\mathbb F}_q}(f(x)),) \]
where ${\rm Tr}_{{\mathbb F}_{q^m}/{\mathbb F}_q}$ denotes the trace map.  In
the special case where $f\in{\mathbb F}_q[x_1,\dots,x_n]$, we can also define
\[ S_m({\mathbb A}^n,f) = \sum_{x\in{\mathbb A}^n({\mathbb F}_{q^m})}
\Psi({\rm Tr}_{{\mathbb F}_{q^m}/{\mathbb F}_q}(f(x))). \]
There are corresponding $L$-functions
\[ L({\mathbb T}^n,f;t) = \exp\bigg(\sum_{m=1}^\infty S_m({\mathbb T}^n,f)
\frac{t^m}{m}\bigg) \]
and 
\[ L({\mathbb A}^n,f;t) = \exp\bigg(\sum_{m=1}^\infty S_m({\mathbb A}^n,f)
\frac{t^m}{m}\bigg). \]

Let ${\mathbb Q}_p$ denote the field of $p$-adic numbers and ${\mathbb Z}_p$
the ring of $p$-adic integers.  Set $\Omega_1 = {\mathbb Q}_p(\zeta_p)$.  Then
$\Omega_1$ is a totally ramified extension of ${\mathbb Q}_p$ of degree
$p-1$.  Let $K$ denote the unramified extension of ${\mathbb Q}_p$ of degree
$r$ and set $\Omega_0 = K(\zeta_p)$.  The Frobenius automorphism $x\mapsto
x^p$ of ${\rm Gal}({\mathbb F}_q/{\mathbb F}_p)$ lifts to a generator $\tau$
of ${\rm Gal}(\Omega_0/\Omega_1)$ by setting $\tau(\zeta_p) = \zeta_p$.  Let
$\Omega$ be the completion of an algebraic closure of $\Omega_0$.  Let ``ord''
denote the additive valuation on $\Omega$ normalized by ${\rm ord}\,p = 1$ and
let ``${\rm ord}_q$'' denote the additive valuation normalized by ${\rm ord}_q
\,q = 1$.  

Let $E(t)=\exp(\sum_{i=0}^\infty t^{p^i}/p^i)$ be the Artin-Hasse exponential
series.  Let $\gamma\in\Omega_1$ be a solution of $\sum_{i=0}^\infty
t^{p^i}/p^i = 0$ satisfying ${\rm ord}\,\gamma = 1/(p-1)$ and set
\[ \theta(t) = E(\gamma t) =\sum_{i=0}^\infty \lambda_it^i\in\Omega_1[[t]]. \]
The series $\theta(t)$ is a splitting function in Dwork's terminology and its
coefficients satisfy
\begin{equation}
{\rm ord}\,\lambda_i\geq \frac{i}{p-1}. 
\end{equation}

Define the {\it Newton polyhedron\/} of $f$, denoted $\Delta(f)$, to be the
convex hull in ${\mathbb R}^n$ of the set $J\cup\{(0,\dots,0)\}$.  Let $C(f)$
be the cone in ${\mathbb R}^n$ over $\Delta(f)$, i.~e., $C(f)$ is the union of
all rays in ${\mathbb R}^n$ emanating from the origin and passing through
$\Delta(f)$.  For any lattice point $u\in C(f)\cap{\mathbb Z}^n$, let $w(u)$,
the {\it weight\/} of $u$, be defined as the smallest positive real number
(necessarily rational) such that $u\in w(u)\Delta(f)$, where $w(u)\Delta(f)$
denotes the dilation of $\Delta(f)$ by the factor $w(u)$.  Then
\[ w:C(f)\cap{\mathbb Z}^n\to \frac{1}{N}{\mathbb Z} \]
for some positive integer $N$.  We fix a choice $\tilde{\gamma}$ of $N$-th
root of $\gamma$ and set $\tilde{\Omega}_0 = \Omega_0(\tilde{\gamma})$,
$\tilde{\Omega}_1 = \Omega_1(\tilde{\gamma})$.  We extend $\tau\in{\rm
  Gal}(\Omega_0/\Omega_1)$ to a generator ${\rm
  Gal}(\tilde{\Omega}_0/\tilde{\Omega}_1)$ by setting $\tau(\tilde{\gamma}) =
\tilde{\gamma}$.  Let $\tilde{\mathcal O}_0$ be the ring of integers of
$\widetilde{\Omega}_0$.  

Let $M$ be a lattice such that $M_J\subseteq M\subseteq {\mathbb Z}^n\cap
{\mathbb R}\langle J\rangle$, let $L = {\rm Hom}_{\mathbb Z}(M,{\mathbb
  Z})$, and let $\ell\in L$.  We extend $\ell$ to a function on ${\mathbb
  Z}^n\cap{\mathbb R}\langle J\rangle$ as follows.  For $u\in{\mathbb Z}^n\cap
{\mathbb R}\langle J\rangle$ we have $p^au\in M$ by (1.2), so we may define 
\[ \ell(u) = p^{-a}\ell(p^au). \]
This definition identifies $L$ with a subgroup of ${\rm Hom}_{\mathbb
  Z}({\mathbb Z}^n\cap {\mathbb R}\langle J\rangle,p^{-a}{\mathbb Z})$.
Define  
\[ M_0(f) = \{u\in{\mathbb Z}^n\cap C(f)\mid {\rm ord}\,\ell(u)\geq 0 \text{
  for all $\ell\in L$.} \} \]
Note that $M_0(f) = M\cap C(f)$.  For $i>0$ let 
\[ M_i(f) = \{u\in{\mathbb Z}^n\cap C(f)\mid \inf_{\ell\in L}\{{\rm
  ord}\,\ell(u)\} = -i\}. \] 
Note that since $L$ has finite rank, the infimum over $L$ always exists.
Furthermore, we have $M_i(f) = \emptyset$ for $i>a$ and
\[ {\mathbb Z}^n\cap C(f) = \bigcup_{i=0}^a M_i(f). \]

We consider the following spaces of power series (where $b\in{\mathbb R}$,
$b\geq 0$, $c\in{\mathbb R}$, and $0\leq i\leq a$):
\begin{align*}
L_i(b,c) &= \bigg\{\sum_{u\in M_i(f)} A_ux^u\mid
  A_u\in\Omega_0,\;{\rm ord}\,A_u\geq bw(u) + c\bigg\} \\
L_i(b) &= \bigcup_{c\in{\mathbb R}} L_i(b,c) \\
B_i &= \bigg\{\sum_{u\in M_i(f)} A_u\tilde{\gamma}^{Nw(u)}x^u\mid
  A_u\in\tilde{\mathcal O}_0,\;A_u\to 0\;{\rm as}\;u\to\infty\bigg\} \\
B'_i &= \bigg\{\sum_{u\in M_i(f)} A_u\tilde{\gamma}^{Nw(u)}x^u\mid
  A_u\in\widetilde{\Omega}_0,\;A_u\to 0\;{\rm as}\;u\to\infty\bigg\}
\end{align*}
We also define $L(b,c)$, $L(b)$, $B$, $B'$ as the unions of these spaces for
$i=0,\dots,a$.  Note that if $b>1/(p-1)$, then $L_i(b)\subseteq B'_i$ and for
$c\geq 0$, $L_i(b,c)\subseteq B_i$.  Similarly $L(b)\subseteq B'$ and for
$c\geq 0$, $L(b,c)\subseteq B$.  Define a norm on $B_i$, $i=0,\dots,a$, as 
follows.  If
\[ \xi = \sum_{u\in M_i(f)}A_u\tilde{\gamma}^{Nw(u)}x^u, \]
then set
\[ \|\xi\| = \sup_{u\in M_i(f)} |A_u|. \]
One defines a norm on $B$ in an analogous fashion.

Let $\hat{f} = \sum_{j\in J}\hat{a}_jx^j$ be the Teichm\"uller lifting of $f$,
i.~e., $\hat{a}_j^q = \hat{a}_j$ and the reduction of $\hat{f}$ modulo $p$
is~$f$.  Set
\begin{align*}
F(x) &= \prod_{j\in J} \theta(\hat{a}_jx^j), \\
F_0(x) &= \prod_{i=0}^{r-1} F^{\tau^i}(x^{p^i}). 
\end{align*}
The estimate (2.1) implies that $F(x)$ and $F_0(x)$ are well-defined and
satisfy
\[ F(x)\in L_0\bigg(\frac{1}{p-1},0\bigg),\quad F_0(x)\in
L_0\bigg(\frac{p}{q(p-1)},0\bigg). \] 

We define the operator $\psi$ on series by
\[ \psi\bigg(\sum_{u\in{\mathbb Z}^n}A_ux^u\bigg) = \sum_{u\in{\mathbb Z}^n} 
A_{pu}x^u.\] 
Clearly, $\psi(L(b,c))\subseteq L(pb,c)$.  

\begin{lemma}
For $1\leq i<a$ we have
\[ \psi(L_i(b,c))\subseteq L_{i+1}(b,c) \]
and for $i=a$ we have
\[ \psi(L_a(b,c)) = 0. \]
Furthermore, the same assertions hold with $L_i(b,c)$ replaced by $B_i'$.
\end{lemma}

\begin{proof}
Let $\ell\in L$ and $pu\in M_i(f)$.  Since ${\rm ord}\,\ell(pu)\geq -i$,
it follows that ${\rm ord}\,\ell(u)\geq -i-1$.  By definition of $M_i(f)$ the
first inequality is an equality for some $\ell\in L$.  The second inequality
is then an equality also for that $\ell$, hence $u\in M_{i+1}(f)$.
\end{proof}

The operator $\alpha = \psi^r\circ F_0$ is an $\widetilde{\Omega}_0$-linear
(resp.\ $\Omega_0$-linear) endomorphism of the space $B'$ (resp.\ $L(b)$ for
$0<b\leq p/(p-1)$).  Furthermore, the operator $\alpha_0 =
\tau^{-1}\circ\psi\circ F_0$ is an $\widetilde{\Omega}_1$-linear (resp.\
$\Omega_1$-linear) endomorphism of $B'$ (resp.\ $L(b)$ for $0<b\leq p/(p-1)$)
and is an $\widetilde{\Omega}_0$-semilinear (resp.\ $\Omega_0$-semilinear)
endomorphism of $B'$ (resp.\ $L(b)$ for $0<b\leq p/(p-1)$).  It follows from
Serre\cite{S} that the operators $\alpha^m$ and $\alpha_0^m$ acting on $B'$
and $L(b)$ for $0<b\leq p/(p-1)$ have well defined traces.  In addition, the
Fredholm determinants $\det(I-t\alpha)$ and $\det(I-t\alpha_0)$ are well
defined and $p$-adically entire.  The Dwork trace formula asserts
\begin{equation}
S_m({\mathbb T}^n,f) = (q^m-1)^n{\rm Tr}(\alpha^m), 
\end{equation}
where $\alpha$ acts either on $B'$ or on some $L(b)$, $0<b\leq p/(p-1)$.  (The
nontrivial additive character implicit on the left-hand side is given by
\[ \Psi(x) = \theta(1)^{{\rm Tr}_{{\mathbb F}_q/{\mathbb F}_p}(x)}. )\]
Let $\delta$ be the operator on formal power series with constant term 1
defined by $g(t)^\delta = g(t)/g(qt)$.  Using the relationship
$\det(I-t\alpha) = \exp(-\sum_{m=1}^\infty {\rm Tr}(\alpha^m)t^m/m)$, equation
(2.3) is equivalent to
\begin{equation}
L({\mathbb T}^n,f;t)^{(-1)^{n-1}} = \det(I-t\alpha)^{\delta^n}.
\end{equation}

Let $\Gamma$ be the map on power series defined by
\[ \Gamma\bigg(\sum_{u\in{\mathbb Z}^n} A_ux^u\bigg) = \sum_{u\in M_0(f)}
A_ux^u. \] 
Define $\tilde{\alpha} = \Gamma\circ \alpha$, an endomorphism of $B_0'$ and
$L_0(b)$ for $0<b\leq p/(p-1)$.  The main technical result of this paper is
the following. 
\begin{theorem}
If $M_J\subseteq M$, then as operator on $B_0'$ and $L_0(b)$ for $0<b\leq
p/(p-1)$ the map $\tilde{\alpha}$ satisfies 
\[ S_m({\mathbb T}^n,f) = (q^m-1)^n{\rm Tr}(\tilde{\alpha}^m). \]
Equivalently, 
\[ L({\mathbb T}^n,f;t)^{(-1)^{n-1}} = \det(I-t\tilde{\alpha})^{\delta^n}. \]
\end{theorem}

\begin{proof}
To fix ideas, we work with the space $B'$.  Note that if $u\in M_0(f)$ and
$v\in M_i(f)$, $1\leq i\leq a$, then $u+v\in M_i(f)$.  This shows that
multiplication by $F$ and $F_0$ are stable on $B_i'$ for $i=1,\dots,a$.
Lemma~2.2 then implies that $\alpha(B_i')\subseteq B_{i+1}'$ for
$i=1,\dots,a-1$ and $\alpha(B'_a) = 0$.  It follows that any power of $\alpha$
acting on $\bigcup_{i=1}^a B'_i$ has trace $0$, so on $\bigcup_{i=1}^a B'_i$
we have $\det(I-t\alpha)=1$.  From \cite[Proposition 9]{S} we then get
\[ \det(I-t\alpha\mid B') = \det(I-t\alpha\mid B'/\bigcup_{i=1}^a B'_i). \]
Under the Banach space isomorphism $B_0'\cong B'/\bigcup_{i=1}^a B'_i$, the
operator $\tilde{\alpha}$ is identified with the operator induced by $\alpha$
on $B'/\bigcup_{i=1}^a B'_i$.  This proves the theorem.
\end{proof}

\section{First applications}

In \cite{AS1,AS2} we made use of the following idea.  First we found $p$-adic
estimates for the entries of the (infinite) matrix of the Frobenius operator
$\alpha$ relative to the basis $\{x^u\mid u\in {\mathbb Z}^n\cap C(f)\}$ for
$L(p/(p-1))$.  These estimates were expressed in terms of the weight function
$w$ (see \cite[Eq.\ (3.8)]{AS1}).  We then used the counting function 
\[ W(k) = {\rm card}\{u\in{\mathbb Z}^n\cap C(f)\mid w(u) = k/N\} \]
to calculate the number of basis elements giving rise to matrix
coefficients having a given $p$-divisibility.  This allowed us to estimate the
$p$-divisibility of the coefficients in the power series $\det(I-t\alpha)$.
Using (2.4) we were then able to deduce information about the exponential sum
and its $L$-function.  

By Theorem 2.5, we can replace the operator $\alpha$ acting
on $L(p/(p-1))$ by the operator $\tilde{\alpha}$ acting on $L_0(p/(p-1))$.
The space $L_0(p/(p-1))$ has basis $\{x^u\mid u\in M_0(f)\}$ and the
corresponding counting function is 
\begin{equation}
W_0(k) = {\rm card}\{u\in M_0(f)\mid w(u) = k/N\}. 
\end{equation}
One can then repeat the arguments of \cite{AS1,AS2} with $W_0(k)$ in place of
$W(k)$, which leads to sharper results.  Typically one would take $M=M_J$ to
get the best estimates.

For example, taking $M=M_J$ and arguing as in \cite[Section 4]{AS1} leads to
the following result, which improves the first inequality of
\cite[Theorem~1.8]{AS1}.  (One can similarly improve the second inequality of
\cite[Theorem~1.8]{AS1}, but we have not worked out the details.)
\begin{theorem}
The following inequality holds:
\[ 0\leq \deg L({\mathbb T}^n,f;t)^{(-1)^{n-1}}\leq n!\,V(f)/[{\mathbb
  Z}^n:M_J], \]
where $V(f)$ denotes the volume of $\Delta(f)$ relative to Lebesgue measure on
${\mathbb R}^n$.  
\end{theorem}

Suppose now that $f\in{\mathbb F}_q[x_1,\dots,x_n]$ and let $\omega(f)$ be the
smallest positive real (hence rational) number such that $\omega(f)\Delta(f)$,
the dilation of $\Delta(f)$ by the factor~$\omega(f)$, contains a point of
$M_J\cap ({\mathbb N}_+)^n$.  We prove the following strengthening of
\cite[Theorem~1.2]{AS2}.
\begin{theorem}
If $f$ is not a polynomial in some proper subset of $\{x_1,\dots,x_n\}$, then
\[ {\rm ord}_q\,S_1({\mathbb A}^n,f)\geq \omega(f). \]
\end{theorem}
As an example of Theorem~3.3, consider the polynomial
\[ f(x_1,x_2) = x_1x_2^4 + x_1^7x_2^3 + x_1^{13}x_2^2. \]
If $p\neq 5$, then $M_J = {\mathbb Z}^2$; so $\omega(f) = 7/25$, which gives
the estimate of \cite[Theorem~1.2]{AS2}.  Theorem~3.3 gives an improvement
when $p=5$.  In this case,
\[ M_J = \{(u_1,u_2)\in{\mathbb Z}^2\mid \text{$u_1+6u_2$ is divisible by
  $25$}\} \]
so $\omega(f) = 1$.

\begin{proof}[Proof of Theorem~$3.3$]

One can repeat the proof given in \cite[Section 4]{AS2}.  The main point to
check is that \cite[Eq.\ (3.13)]{AS2} still holds when we replace the $p$-adic
Banach space used there (namely, $L(p/(p-1))$) by the Banach space
$L_0(p/(p-1))$ corresponding to the choice of lattice $M=M_J$.  It then
follows that \cite[Eq.~(3.14)]{AS2} holds when using the Banach space
$L_0(p/(p-1))$.  This allows one to repeat the argument of
\cite[Section~4]{AS2} {\it mutatis mutandis}.  

For any subset $A\subseteq\{1,\dots,n\}$, let $f_A$ be the polynomial obtained
from $f$ by setting $x_i=0$ for $i\in A$.  Define
\[ L_{0,(A)}\bigg(\frac{p}{p-1}\bigg) = \bigg\{ \sum_{u\in M_0(f)} A_ux^u\in
  L_0\bigg(\frac{p}{p-1}\bigg)\bigg| \text{$A_u=0$ if $u_i\neq 0$ for some
 $i\in A$}\bigg\}. \] 
Let $\Gamma_A:L_0(p/(p-1))\to L_{0,(A)}(p/(p-1))$ be defined by
\[ \Gamma_A\bigg(\sum_{u\in M_0(f)} A_ux^u\bigg) = \sum_{\substack{u\in
 M_0(f)\\ \text{$u_i=0$ for $i\in A$}}} A_ux^u \]
 and let $\tilde{\alpha}_A$ denote the endomorphism $\Gamma_A\circ
\tilde{\alpha}$ of $L_{0,(A)}(p/(p-1))$.  We need to check that (see
 \cite[Eq.\ (3.13)]{AS2}) 
\begin{equation}
S_m({\mathbb T}^{n-|A|},f_A) = (q^m-1)^{n-|A|}{\rm Tr}(\tilde{\alpha}_A\mid 
L_{0,(A)}(p/(p-1))), 
\end{equation}
where $|A|$ denotes the cardinality of $A$.  Let $J_A\subseteq{\mathbb
  Z}^{n-|A|}$ denote the set of exponents of $f_A$:
\[ f_A = \sum_{j\in J_A} a_jx^j\in{\mathbb F}_q[\{x_i\}_{i\not\in A}]. \]
The power series in $L_{0,(A)}(p/(p-1))$ have exponents in the
lattice $M_J\cap{\mathbb R}^{n-|A|}$.  By Theorem~2.5, Eq.~(3.4) will hold
provided $M_{J_A}\subseteq M_J\cap{\mathbb R}^{n-|A|}$.  But this is clear.  
\end{proof}

We derive a generalization of Theorem~1.3 from Theorem~3.3.  Let
$f_1,\dots,f_r\in{\mathbb F}_q[x_1,\dots,x_n]$ and let $N(f_1,\dots,f_r)$
denote the number of solutions in ${\mathbb F}_q$ to the system
$f_1=\dots=f_r=0$.  Let $y_1,\dots,y_r$ be additional variables and set
\[ F = \sum_{i=1}^r y_if_i\in{\mathbb F}_q[x_1,\dots,x_n,y_1,\dots,y_r]. \]
It is easily seen that
\[ S_1({\mathbb A}^{n+r},F) = q^rN(f_1,\dots,f_r). \]
Applying Theorem~3.3 to $F$ gives the following result, of which Theorem~1.3
is the special case $r=1$.
\begin{corollary}
${\rm ord}_q\,N(f_1,\dots,f_r)\geq\omega(F)-r$.
\end{corollary}

\section{Nondegeneracy relative to a lattice}

The results of \cite{AS3,AS4} are cohomological in nature and require a more
detailed development.  Suppose that ${\mathbb Z}\langle J\rangle$ has rank
$k$.  Let $M$ be a lattice, ${\mathbb Z}\langle
J\rangle\subseteq M\subseteq{\mathbb Z}^n\cap{\mathbb R}\langle J\rangle$, and
choose a basis of linear forms 
$\{\ell_i\}_{i=1}^k$ for $L = {\rm Hom}_{\mathbb Z}(M,{\mathbb Z})$.  We
define ``differential operators'' $\{E_{\ell_i}\}_{i=1}^k$ on the ring
${\mathbb F}_q[x^u\mid u\in M]$ by linearity and the formula 
\[ E_{\ell_i}(x^u) = \ell_i(u)x^u. \]
This definition is motivated by the fact that if we write
\[ \ell_i(u_1,\dots,u_n) = \sum_{j=1}^n a_{ij}u_j, \]
where $u=(u_1,\dots,u_n)\in M\subseteq{\mathbb Z}^n$ and the $a_{ij}$ are
rational numbers, and put $E_{\ell_i} = \sum_{j=1}^n
a_{ij}x_j\partial/\partial x_j$, then in characteristic $0$
\[ E_{\ell_i}(x^u) = \sum_{j=1}^n a_{ij}x_j\frac{\partial}{\partial x_j}(x^u)
= \ell_i(u)x^u. \]

Let $f$ be given by (1.1) and let $\sigma$ be a subset of $\Delta(f)$.  Define
\[ f_\sigma =\sum_{j\in J\cap\sigma} a_jx^j. \]
We say that $f$ is {\it nondegenerate relative to\/} $(\Delta(f),M)$ if for
every face $\sigma$ of $\Delta(f)$ that does not contain the origin, the
Laurent polynomials $\{E_{\ell_i}(f_\sigma)\}_{i=1}^k$ have no common zero in
$(\bar{\mathbb F}_q^\times)^n$, where $\bar{\mathbb F}_q$ denotes an algebraic
closure of ${\mathbb F}_q$.  Note that the condition ${\mathbb Z}\langle
J\rangle\subseteq M$ guarantees that all $f_\sigma$ lie in ${\mathbb
  F}_q[x^u\mid u\in M]$, so the $E_{\ell_i}(f_\sigma)$ are defined.  Note also
that this definition depends only on $M$ and not on the choice of basis
$\{\ell_i\}_{i=1}^k$ for $L$: any two bases for $L$ are related by a matrix in
${\rm GL}(k,{\mathbb Z})$.  (We remark that this idea to replace the
differential operators $x_i\partial/\partial x_i$ by certain linear
combinations with coefficients that are not $p$-integral appears in
nascent form in Dwork\cite{DW1}, where it was needed to calculate the
$p$-adic cohomology of smooth hypersurfaces of degree divisible by $p$.)

The condition used in \cite{AS3}, that $f$ be ``nondegenerate relative to
$\Delta(f)$,'' is equivalent to the condition that $f$ be nondegenerate
relative to $(\Delta(f),{\mathbb Z}^n\cap{\mathbb R}\langle J\rangle)$ in the
sense of the present definition.  We 
make the relationship between this definition and our earlier one more
explicit.  There is a basis ${\bf e}_1,\dots,{\bf e}_n$ for ${\mathbb Z}^n$
and positive integers $d_1,\dots,d_k$, $k\leq n$, such that $d_1{\bf
  e}_1,\dots,d_k{\bf e}_k$ is a basis for $M$.  After a coordinate change
on~${\mathbb T}^n$, we may take ${\bf e}_1,\dots,{\bf e}_n$ to be the standard
basis for ${\mathbb Z}^n$.  This implies that there exists a Laurent
polynomial 
\[ g = \sum_{c\in C} b_cx^c\in{\mathbb F}_q[x_1^{\pm 1},\dots,x_k^{\pm 1}], \]
where $C$ is a finite subset of ${\mathbb Z}^k$, such that
\begin{equation}
f(x_1,\dots,x_n) = g(x_1^{d_1},\dots,x_k^{d_k}).
\end{equation}
Note that (4.1) implies
\begin{equation}
[{\mathbb Z}\langle C\rangle:{\mathbb Z}\langle J\rangle] = d_1\cdots d_k\; ( =
[{\mathbb Z}^n\cap{\mathbb R}\langle J\rangle:M]).
\end{equation}

{\it Remark}.  When we choose $M=M_J$, it follows from (1.2) that each
$d_i$ is a power of $p$.  In this case, the exponential sums
associated to $f$ and $g$ are identical.

\begin{proposition}
The Laurent polynomial $f$ is nondegenerate relative to $(\Delta(f),M)$ if
and only if $g$ is nondegenerate relative to $(\Delta(g),{\mathbb Z}^k)$.
\end{proposition}

\begin{proof}
Equation (4.1) implies that there is a one-to-one correspondence between the
faces of $\Delta(f)$ and the faces of $\Delta(g)$.  Specifically, the face
$\sigma$ of $\Delta(f)$ corresponds to the face $\sigma'$ of $\Delta(g)$
defined by
\[ \sigma' = \{(d_1^{-1}u_1,\dots,d_k^{-1}u_k)\in{\mathbb R}^k\mid
(u_1,\dots,u_k)\in\sigma\}. \]
Furthermore, we have 
\[ f_\sigma(x_1,\dots,x_k) = g_{\sigma'}(x_1^{d_1},\dots,x_k^{d_k}). \]
Using $u_1,\dots,u_k$ as coordinates on ${\mathbb Z}^k$, we may take as basis
for ${\rm Hom}_{\mathbb Z}({\mathbb Z}^k,{\mathbb Z})$ the linear forms
$\{\ell'_i\}_{i=1}^k$ defined by
\[ \ell'_i(u_1,\dots,u_k) = u_i \]
and we may take as basis for $L = {\rm Hom}_{\mathbb Z}(M,{\mathbb Z})$ the
linear forms $\{\ell_i\}_{i=1}^k$ defined by
\[ \ell_i(u_1,\dots,u_k) = d_i^{-1}u_i. \]
It is straightforward to check that for $i=1,\dots,k$, 
\[ E_{\ell_i}(f_\sigma)(x_1,\dots,x_k) =
E_{\ell'_i}(g_{\sigma'})(x_1^{d_1},\dots,x_k^{d_k}). \] 
This implies the proposition.
\end{proof}

\begin{lemma}
Put $[{\mathbb Z}^n\cap{\mathbb R}\langle J\rangle:M_J] = p^a$ and let
$M\subseteq{\mathbb Z}^n\cap{\mathbb R}\langle J\rangle$ be a lattice 
containing ${\mathbb Z}\langle J\rangle$.  Then $M\subseteq M_J$ if and only
if $p^a\mid [{\mathbb Z}^n\cap{\mathbb R}\langle J\rangle:M]$.
\end{lemma}

\begin{proof}
Suppose that $p^a\mid [{\mathbb Z}^n\cap{\mathbb R}\langle J\rangle:M]$.  Then
$[M:{\mathbb Z}\langle J\rangle] = e'$ with $(e',p) = 1$.  In particular,
$e'm\in {\mathbb Z}\langle J\rangle$ for all $m\in M$, so $M\subseteq M_J$.
The other direction of the assertion is clear.
\end{proof}

There are restrictions on the lattices with respect to which $f$ can be
nondegenerate. 
\begin{proposition}
Let $M$ be a lattice, ${\mathbb Z}\langle J\rangle\subseteq M\subseteq{\mathbb
  Z}^n\cap{\mathbb R}\langle J\rangle$. \\ 
{\bf (a)} If $f$ is nondegenerate relative to $(\Delta(f),M)$, then $M\subseteq
M_J$.  \\
{\bf (b)} Suppose $M\subseteq M_J$.  Then $f$ is nondegenerate relative to
$(\Delta(f),M)$ if and only if $f$ is nondegenerate relative to
$(\Delta(f),M_J)$.  
\end{proposition}

\begin{proof}
We may assume without loss of generality that ${\mathbb Z}\langle J\rangle$ is
a subgroup of ${\mathbb Z}^n$ of rank $n$.  For if ${\rm rank}({\mathbb
  Z}\langle J\rangle) = k<n$, then by (4.1) we may take $f$ to be a Laurent
polynomial in $x_1,\dots,x_k$, in which case ${\mathbb Z}\langle J\rangle$ is
a subgroup of ${\mathbb Z}^n\cap{\mathbb R}\langle J\rangle (={\mathbb Z}^k)$
of rank~$k$. 

We suppose $M$ is not contained in $M_J$ and prove that $f$ must be
degenerate relative to $(\Delta(f),M)$.  By (4.2) and Lemma 4.4, we
have $p^a\nmid [{\mathbb Z}\langle C\rangle:{\mathbb Z}\langle
  J\rangle]$.  But $p^a\mid[{\mathbb Z}^n:{\mathbb Z}\langle
  J\rangle]$, so $p\mid[{\mathbb Z}^n:{\mathbb Z}\langle C\rangle]$.
Arguing as in the proof of (4.1) then shows that there exists a
Laurent polynomial  
\[ h = \sum_{e\in E}c_ex^e\in{\mathbb F}_q[x_1^{\pm1},\dots,x_n^{\pm
  1}] \]
such that
\begin{equation}
g(x_1,\dots,x_n) = h(x_1,\dots,x_{n-1},x_n^p).
\end{equation}
To show $f$ is degenerate relative to $(\Delta(f),M)$, it suffices by
Proposition~4.3 to show that any Laurent polynomial $g$ of the form (4.6) is
degenerate relative to $(\Delta(g),{\mathbb Z}^n)$.

We must find a face $\sigma$ of $\Delta(g)$ not containing the origin such
that $\{x_i\partial g_\sigma/\partial x_i\}_{i=1}^n$ have a common zero in
$(\bar{\mathbb F}_q^\times)^n$.  Note that (4.6) implies that all $x_n\partial
g_\sigma/\partial x_n$ vanish identically.  We assume that for every face
$\sigma$ of $\Delta(g)$ not containing the origin and having codimension $>1$,
the Laurent polynomials $\{x_i\partial g_\sigma/\partial x_i\}_{i=1}^{n-1}$
have no common zero in $(\bar{\mathbb F}_q^\times)^n$.  We then prove that for
every face $\sigma$ of $\Delta(g)$ not containing the origin and having
codimension $1$, the Laurent polynomials $\{x_i\partial g_\sigma/\partial
x_i\}_{i=1}^{n-1}$ do have a common zero in $(\bar{\mathbb F}_q^\times)^n$.

Fix such a face $\sigma$ of codimension $1$.  By our hypothesis, for every
proper face $\tau$ of $\sigma$, 
\begin{equation} 
\text{ $\{x_i\partial g_{\tau}/\partial x_i\}_{i=1}^{n-1}$ have no common zero
  in $(\bar{\mathbb F}_q^\times)^n$. }
\end{equation}
We make a change of variable in order to apply a theorem of
Kouchnirenko\cite{KO}.  First of all, the face $\sigma$ lies in a unique
hyperplane $H$ in ${\mathbb R}^n$.  Choose $\alpha\in{\mathbb Z}^n\cap H$ and
set 
\[ \phi_i = x^{-\alpha}x_i\frac{\partial
  g_\sigma}{\partial x_i} \]
for $i=1,\dots,n-1$.  Put $K=-\alpha + (J\cap\tau)$.  Then each $\phi_i$ can be
written in the form 
\begin{equation}
\phi_i = \sum_{k\in K} a^{(i)}_kx^k. 
\end{equation}
Note that $K$ is contained in the hyperplane $-\alpha + H$, which contains the
origin.  Choose a basis ${\bf b}^{(1)},\dots,{\bf b}^{(n-1)}$ for the
rank-$(n-1)$ lattice ${\mathbb Z}^n\cap(-\alpha+H)$ and set $y_i = x^{{\bf
    b}^{(i)}}$.  Let $B$ be the $n\times(n-1)$-matrix whose columns are the
${\bf b}^{(j)}$.  Multiplication of vectors by $B$ is an isomorphism from
${\mathbb R}^{n-1}$ onto the subspace $-\alpha+H$ of ${\mathbb R}^n$ which
induces an isomorphism of the lattice ${\mathbb Z}^{n-1}$ onto the lattice
${\mathbb Z}^n\cap(-\alpha+H)$.  From (4.8) we have
\begin{equation}
\phi_i = \sum_{v\in V}a^{(i)}_{Bv}x^{Bv} = \sum_{v\in V}a^{(i)}_{Bv}y^v, 
\end{equation}
where $V = B^{-1}(K)\subseteq{\mathbb Z}^{n-1}$.  

The convex hull of $K$ is $-\alpha+\sigma$, so the convex hull of $V$ is
$B^{-1}(-\alpha+\sigma)$.  For any face $\tau$ of $\sigma$, there are
corresponding faces $-\alpha+\tau$ of the convex hull of $K$ and
$B^{-1}(-\alpha + \tau)$ of the convex hull of $V$.  Set
\[ (\phi_i)_{B^{-1}(-\alpha+\tau)} = \sum_{v\in V\cap B^{-1}(-\alpha+\tau)}
a^{(i)}_{Bv}y^v. \]
By (4.7), the $\{(\phi_i)_{B^{-1}(-\alpha+\tau)}\}_{i=1}^{n-1}$ have no common
zero in $(\bar{\mathbb F}_q^\times)^{n-1}$.  It then follows from
\cite[1.18~Th\'eor\`eme~${\rm III}'$(ii)]{KO} that the number of common zeros
$(y_1,\dots,y_{n-1})$ of $\{\phi_i\}_{i=1}^{n-1}$ in $(\bar{\mathbb
  F}_q^\times)^{n-1}$ equals $(n-1)!$ times the $(n-1)$-volume of the convex
hull of~$V$.  In particular, the set of common zeros in $(\bar{\mathbb
  F}_q^\times)^{n-1}$ is nonempty.

Choose an $(n-1)\times n$ integral matrix $B'=(b'_{ij})$ such that $B'B =
I_{n-1}$.  For any common zero $(y_1,\dots,y_{n-1})\in(\bar{\mathbb
  F}_q^\times)^{n-1}$, set 
\[ x_j = y_1^{b'_{1j}}\cdots y_{n-1}^{b'_{n-1,j}} \]
for $j=1,\dots,n$.  This gives $x\in(\bar{\mathbb F}_q^\times)^n$ satisfying
$y_i = x^{{\bf b}^{(i)}}$ for $i=1,\dots,n-1$, so this $x$ is a common zero of
$\{x_i\partial g_\sigma/\partial x_i\}_{i=1}^{n-1}$.  Thus $g$ is degenerate
relative to $(\Delta(g),{\mathbb Z}^n)$, which establishes part (a) of the
proposition.  

Now suppose that $M\subseteq M_J$.  Choose a basis $\{{\bf e}^{(i)}\}_{i=1}^n$
for $M_J$ and integers $d_1,\dots,d_n$ such that $\{d_i{\bf
  e}^{(i)}\}_{i=1}^n$ is a basis for $M$.  By Lemma~4.4, $p\nmid d_1\cdots
d_n$.  Let $\{\ell_i\}_{i=1}^n$ be the basis for ${\rm Hom}_{\mathbb
  Z}(M_J,{\mathbb Z})$ defined by
\[ \ell_i({\bf e}^{(j)}) = \delta_{ij}\quad(\text{Kronecker's delta}). \]
Then $\{d_i^{-1}\ell_i\}_{i=1}^n$ is a basis for ${\rm Hom}_{\mathbb
  Z}(M,{\mathbb Z})$.  And since $(d_i,p) = 1$ for all $i$,
the $\{E_{\ell_i}(f_\sigma)\}_{i=1}^n$ have no common zero in $(\bar{\mathbb
  F}_q^\times)^n$ if and only if the same is true of the
  $\{E_{d_i^{-1}\ell_i}(f_\sigma)\}_{i=1}^n$.  This establishes part (b) of
the proposition. 
\end{proof}

By Proposition 4.5(a), we must have $M\subseteq M_J$ if we hope to have $f$
nondegenerate relative to $(\Delta(f),M)$.  On the other hand, we must have
$M_J\subseteq M$ in order for the trace formula (Theorem~2.5) to hold for $M$.
Thus the only practical choice for $M$ is to take $M=M_J$.  Recall from
Section~2 that if $g(t)$ is a power series with constant term 1, then
$g(t)^\delta = g(t)/g(qt)$.

\begin{theorem}
Suppose that ${\mathbb Z}\langle J\rangle$ has rank $k$ and that $f$ is
nondegenerate relative to $(\Delta(f),M_J)$.  Then $L({\mathbb
  T}^n,f;t)^{(-1)^{n-1}} = P(t)^{\delta^{n-k}}$, where $P(t)$ is a polynomial 
of degree $k!\,V_{M_J}(f)$ and $V_{M_J}(f)$ denotes the volume of
$\Delta(f)$ relative to Lebesgue measure on ${\mathbb R}\langle
J\rangle$ normalized so that a fundamental domain for $M_J$ has volume~$1$.
\end{theorem}

\begin{proof}
One repeats the arguments of \cite{AS3} with the modifications
introduced for Theorem~2.5: replace $L(b)$ and $B'$ by $L_0(b)$ and $B_0'$ and
use $\tilde{\alpha}$ in place of $\alpha$.  We recall some of these details as
they are needed in the proof of Theorem~4.20 below.   

Let 
\[ \Omega^\bullet: 0\to\Omega^0\to\dots\to\Omega^n\to 0 \]
be the cohomological Koszul complex on $B_0'$ defined by the differential
operators $\{\hat{D}_i\}_{i=1}^n$ constructed in \cite[Section 2]{AS3}.  The
endomorphism $\tilde{\alpha}$ of $B_0'$ constructed in Section~2 can be
extended to an endomorphism $\tilde{\alpha}_\bullet$ of the complex
$\Omega^\bullet$ by noting that $\Omega^i = (B_0')^{\binom{n}{i}}$ and then
defining $\tilde{\alpha}_i:\Omega^i\to\Omega^i$ to be
\begin{equation}
(q^{n-i}\tilde{\alpha})^{\binom{n}{i}}:(B_0')^{\binom{n}{i}}\to
(B_0')^{\binom{n}{i}}. 
\end{equation}
Theorem~2.5 is equivalent to the assertion that
\[ L({\mathbb T}^n,f;t) = \prod_{i=0}^n
\det(I-t\tilde{\alpha}_i\mid\Omega^i)^{(-1)^{i+1}}, \]
which implies that 
\begin{equation}
L({\mathbb T}^n,f;t) = \prod_{i=0}^n
\det(I-t\tilde{\alpha}_i\mid H^i(\Omega^\bullet))^{(-1)^{i+1}}. 
\end{equation}

Put $R = {\mathbb F}_q[x^u\mid u\in M_0(f)]$.  The ring $R$ has an increasing
filtration defined by the weight function $w$ of Section 2: $F_{i/N}R$ is the
subspace spanned by $\{x^u\mid w(u)\leq i/N\}$.  Let $\bar{R} =
\bigoplus_{i=0}^\infty \bar{R}_{i/N}$ be the associated graded ring, i.~e.,
$\bar{R}_{i/N} = F_{i/N}R/F_{(i-1)/N}$.  Now suppose that $f$ is nondegenerate
relative to $(\Delta(f),M_J)$, let $\{\ell_i\}_{i=1}^k$ be a basis for
$L={\rm Hom}_{\mathbb Z}(M_J,{\mathbb Z})$, and let
$\overline{E_{\ell_i}(f)}\in\bar{R}_1$ be the image in the associated graded
ring of $E_{\ell_i}(f)\in F_1R$.  The nondegeneracy hypothesis implies by the
arguments of \cite{KO} that $\{\overline{E_{\ell_i}(f)}\}_{i=1}^k$ is a
regular sequence in~$\bar{R}$, i.~e.,  the (cohomological) Koszul complex on
$\bar{R}$ defined by $\{\overline{E_{\ell_i}(f)}\}_{i=1}^k$ has vanishing
cohomology except in top dimension.  Furthermore, also by the methods of
\cite{KO}, one can show that the single nonvanishing cohomology group has
dimension $k!\,V_{M_J}(f)$.  

Since $M_J\subseteq{\mathbb Z}^n$, we may express the elements of $L$ as
linear forms in $n$ variables.  Write
\[ \ell_i(u_1,\dots,u_n) = \sum_{j=1}^n a_{ij}u_j,\quad a_{ij}\in
  p^{-a}{\mathbb Z}. \]
Put $\hat{D}_{\ell_i} = \sum_{j=1}^n a_{ij}\hat{D}_j$ and let
$\Omega^\bullet_\ell$ be the cohomological Koszul complex on $B_0'$ defined by
$\{\hat{D}_{\ell_i}\}_{i=1}^k$.  The Frobenius action
$\tilde{\alpha}_i:\Omega^i_\ell\to \Omega^i_\ell$ is defined to be
\[ (q^{k-i}\tilde{\alpha})^{\binom{k}{i}}:(B'_0)^{\binom{k}{i}}\to
  (B'_0)^{\binom{k}{i}}. \]
The ``reduction mod $p$'' (see \cite[Lemma~2.10]{AS3}) of
$\Omega^\bullet_\ell$ is the Koszul complex on $\bar{R}$ defined by
$\{\overline{E_{\ell_i}(f)}\}_{i=1}^k$.  The cohomological lifting 
theorem of Monsky (see \cite[Theorem~8.5]{M} or \cite[Theorem~A.1]{AS3}) then
implies that the cohomology of $\Omega^\bullet_\ell$ vanishes except in top
dimension and that $H^k(\Omega^\bullet_\ell)$ has dimension
$k!\,V_{M_J}(f)$.  But since $\{\hat{D}_{\ell_i}\}_{i=1}^k$ are linear
combinations of $\{\hat{D}_i\}_{i=1}^n$ and vice versa, it follows
that (as Frobenius modules) 
\[ H^i(\Omega^\bullet)\cong (H^k(\Omega^\bullet_\ell))^{\binom{n-k}{n-i}}, \]
where it is understood that the right-hand side vanishes if $i<k$.  In
particular we have $H^n(\Omega^\bullet)\cong H^k(\Omega^\bullet_\ell)$, hence
\[ \det(I-t\tilde{\alpha}_i\mid H^i(\Omega^\bullet)) =
\det(I-q^{n-i}t\tilde{\alpha}_n\mid H^n(\Omega^\bullet))^{\binom{n-k}{n-i}}. \]
From equation (4.12) we then get
\begin{equation}
L({\mathbb T}^n,f;t) = \prod_{i=k}^n \det(I-q^{n-i}t\tilde{\alpha}_n\mid
H^n(\Omega^\bullet))^{(-1)^{i+1}\binom{n-k}{n-i}}.
\end{equation}
If we put
\[ P(t) = \det(I-t\tilde{\alpha}_n\mid H^n(\Omega^\bullet))\;( =
\det(I-t\tilde{\alpha}_k\mid H^k(\Omega^\bullet_\ell))), \]
then $P(t)$ is a polynomial of degree $k!\,V_{M_J}(f)$ and (4.13)
implies that 
\[ L({\mathbb T}^n,f;t)^{(-1)^{n-1}} = P(t)^{\delta^{n-k}}. \]
This completes the proof of Theorem~4.10.  
\end{proof}

Assume the hypotheses of Theorem 4.10.  
The quotient ring
\[ \bar{R}/(\overline{E_{\ell_1}(f)},\dots,\overline{E_{\ell_k}(f)}) \]
is a graded ring of dimension $k!\,V_{M_J}(f)$ over ${\mathbb F}_q$.  Put 
\[ a_i = \dim_{{\mathbb F}_q}
(\bar{R}/(\overline{E_{\ell_1}(f)},\dots,\overline{E_{\ell_k}(f)}))_{i/N}. \] 
One can show that $a_i = 0$ for $i>kN$.  By either repeating the argument of
\cite{AS3} or replacing the polynomial $f$ by the polynomial
$g(x_1^{e_1},\dots,x_n^{e_n})$ constructed in the Introduction and applying
\cite[Theorem~3.10]{AS3}, one obtains the following generalization of part of
\cite[Theorem~3.10]{AS3}. 

\begin{theorem}
Under the hypotheses of Theorem $4.10$, the Newton polygon of the polynomial
$P(t)$ relative to the valuation ${\rm ord}_q$ lies on or above the Newton
polygon relative to ${\rm ord}_q$ of the polynomial $\prod_{i=0}^{kN}
(1-q^{i/N}t)^{a_i}$. 
\end{theorem}

{\it Remark}.  We recall the combinatorial description of the $a_i$.
Take $M=M_J$ in (3.1) and form the generating series
\[ H(t) = \sum_{i=0}^\infty W_0(i) t^{i/N}. \]
Then
\[ H(t) = \frac{\sum_{i=0}^{kN} a_i t^{i/N}}{(1-t)^k}. \]

We generalize Theorem~4.10 to the affine case.  (The corresponding
generalization of Theorem~4.14 is somewhat more involved so we
postpone that to a future article.)  Let 
\[ f = \sum_{j\in J} a_jx^j\in{\mathbb F}_q[x_1^{\pm 1},\dots,x_k^{\pm
  1},x_{k+1},\dots,x_n]. \]
For each subset $A\subseteq\{k+1,\dots,n\}$, let
$f_A$ be the polynomial obtained from $f$ by setting $x_i = 0$ for all $i\in
A$.  Then
\begin{equation}
f_A = \sum_{j\in J_A} a_jx^j\in{\mathbb F}[x_1^{\pm 1},\dots,x_k^{\pm
  1},\{x_i\}_{i\not\in A}],  
\end{equation}
where $J_A = \{j=(j_1,\dots,j_n)\in J\mid \text{$j_i = 0$ for $i\in A$}\}$.
We call $f$ {\it convenient\/} if for each such $A$ one has
\[ \dim\Delta(f_A) = \dim\Delta(f) - |A|. \]
Suppose $f$ is convenient and nondegenerate relative
to $(\Delta(f),M_J)$.  The hypothesis that $f$ be convenient guarantees that
$f_A$ is also convenient, and the hypothesis that $f$ be nondegenerate
relative to $(\Delta(f),M_J)$ implies that $f_A$ is nondegenerate relative to
$(\Delta(f_A),M_J\cap{\mathbb R}\langle J_A\rangle)$.  By Proposition~4.5(a),
we must then have  $M_J\cap{\mathbb R}\langle J_A\rangle\subseteq M_{J_A}$.
The reverse inclusion is clear, so
\begin{equation}
M_{J_A} = M_J\cap{\mathbb R}\langle J_A\rangle
\end{equation}
and we conclude that $f_A$ is nondegenerate relative to
$(\Delta(f_A),M_{J_A})$.  Applying Theorem~4.10, we get that 
\begin{equation}
L({\mathbb T}^{n-|A|},f_A;t)^{(-1)^{n-|A|-1}} = P_A(t)^{\delta^{n -
    \dim\Delta(f)}}
\end{equation}
where $P_A(t)$ is a polynomial of degree 
\begin{equation}
\deg P_A(t) = (\dim\Delta(f_A))!\,V_{M_{J_A}}(f_A).
\end{equation}

The standard toric decomposition of affine space  gives
\[ S_m({\mathbb T}^k\times{\mathbb A}^{n-k},f) =
\sum_{A\subseteq\{k+1,\dots,n\}} S_m({\mathbb T}^{n-|A|},f_A), \]
hence
\begin{equation}
L({\mathbb T}^k\times{\mathbb A}^{n-k},f;t)^{(-1)^{n-1}} =
\prod_{A\subseteq\{k+1,\dots,n\}} 
(L({\mathbb T}^{n-|A|},f_A;t)^{(-1)^{n-|A|-1}})^{(-1)^{|A|}}.
\end{equation}
Put
\[ \nu(f) = \sum_{A\subseteq\{k+1,\dots,n\}} (-1)^{|A|}(\dim\Delta(f_A))!
\,V_{M_{J_A}}(f_A). \] 
\begin{theorem}
If $f\in{\mathbb F}_q[x_1^{\pm 1},\dots,x_k^{\pm 1},x_{k+1},\dots,x_n]$ is
nondegenerate relative to $(\Delta(f),M_J)$ and convenient, then 
\begin{equation}
L({\mathbb T}^k\times{\mathbb A}^{n-k},f;t)^{(-1)^{n-1}} =
Q(t)^{\delta^{n-\dim\Delta(f)}},
\end{equation}
where $Q(t)$ is a polynomial of degree $\nu(f)$.
\end{theorem}

\begin{proof}
It follows from (4.17) and (4.19) that (4.21) holds with
\begin{equation}
Q(t) = \prod_{A\subseteq\{k+1,\dots,n\}} P_A(t)^{(-1)^{|A|}},
\end{equation}
a rational function of degree $\nu(f)$ by (4.18).  It remains only to show
that $Q(t)$ is a polynomial.

In the proof of Theorem~4.10, we constructed a complex $\Omega^\bullet$
satisfying 
\begin{equation}
H^i(\Omega^\bullet)\cong (H^n(\Omega^\bullet))^{\binom{n-\dim\Delta(f)}{n-i}}
\end{equation}
and $L({\mathbb T}^n,f;t)^{(-1)^{n-1}} = P(t)^{\delta^{n-\dim\Delta(f)}}$,
where
\begin{equation}
P(t) = \det(I-t\tilde{\alpha}_n\mid H^n(\Omega^\bullet)).
\end{equation}
Since $f$ is nondegenerate and convenient, each of the polynomials $f_A$
satisfies the hypotheses of that theorem, so analogous assertions are true.
Let 
\[ \Omega_A^\bullet: 0\to\Omega^0_A\to\dots\to\Omega_A^{n-|A|}\to 0 \]
be the corresponding cohomological Koszul complex with differential operators
$\{\hat{D}_i^A\}_{i\not\in A}$ and Frobenius operators
$\{\tilde{\alpha}^A_i\}_{i=0}^{n-|A|}$.  We have  
\begin{equation}
H^i(\Omega^\bullet_A) =
(H^{n-|A|}(\Omega^\bullet_A))^{\binom{n-\dim\Delta(f)}{n-|A|-i}} 
\end{equation}
and $L({\mathbb T}^{n-|A|},f_A;t)^{(-1)^{n-|A|-1}} =
P_A(t)^{n-\dim\Delta(f)}$, where
\begin{equation}
P_A(t) = \det(I-t\tilde{\alpha}^A_{n-|A|}\mid H^{n-|A|}(\Omega^\bullet_A)).
\end{equation}

There is an exact sequence of complexes (see
Libgober-Sperber\cite[Eq.~(4.1)]{LS}) 
\[ \Omega^\bullet\to \bigoplus_{|A|=1} \Omega^\bullet_A[-1]
\to \bigoplus_{|A|=2} \Omega^\bullet_A[-2]\to\dots\to
\Omega^\bullet_{\{k+1,\dots,n\}}[-n+k]\to 0. \]
Let $\bar{\Omega}^\bullet = \ker(\Omega^\bullet\to \bigoplus_{|A|=1}
\Omega_A^\bullet[-1])$, so that there is an exact sequence
\begin{equation}
0\to\bar{\Omega}^\bullet\to \Omega^\bullet\to \bigoplus_{|A|=1}
\Omega^\bullet_A[-1]\to\dots\to
\Omega^\bullet_{\{k+1,\dots,n\}}[-n+k]\to 0. 
\end{equation}
Equations (4.23), (4.25), (4.27),  and induction on $n-k$ show that
\begin{equation}
H^i(\bar{\Omega}^\bullet)\cong
(H^n(\bar{\Omega}^\bullet))^{\binom{n-\dim\Delta(f)}{n-i}}.
\end{equation}
Equation (4.27) implies that
\begin{multline}
\prod_{i=0}^n \det(I-t\tilde{\alpha}_i\mid
H^i(\bar{\Omega}^\bullet))^{(-1)^{i+1}} = \\
\prod_{A\subseteq\{k+1,\dots,n\}}
\biggl( \prod_{i=0}^{n-|A|} \det(I-t\tilde{\alpha}^A_i\mid
H^i(\Omega_A^\bullet))^{(-1)^{i+|A|+1}} \biggr)^{(-1)^{|A|}}. 
\end{multline}
The inner product on the right-hand side of (4.29) equals $L({\mathbb
  T}^{n-|A|},f_A,t)^{(-1)^{|A|}}$, hence by (4.19) the right-hand side equals
$L({\mathbb T}^k\times{\mathbb A}^{n-k},f;t)$.  By (4.28) the left-hand side
equals 
\[ \prod_{i=0}^{n-\dim\Delta(f)} \det(I-tq^i\tilde{\alpha}_n \mid
  H^n(\bar{\Omega}^\bullet))^{(-1)^{n-1}\binom{n-\dim\Delta(f)}{i}}. \]
We thus have
\[ L({\mathbb T}^k\times{\mathbb A}^{n-k},f;t)^{(-1)^{n-1}} =
\det(I-t\tilde{\alpha}_n\mid
H^n(\bar{\Omega}^\bullet))^{\delta^{n-\dim\Delta(f)}}. \]
Comparison with (4.21) then shows that
\[ Q(t) = \det(I-t\tilde{\alpha}_n\mid H^n(\bar{\Omega}^\bullet)), \]
hence $Q(t)$ is a polynomial.
\end{proof}

We explain how to compute the archimedian absolute values of
the roots of the polynomial $Q(t)$ under the hypothesis of
Theorem~4.20.  Take $M=M_J$ and let $g$ be the Laurent polynomial
associated to $f$ by (4.1).  As noted in the proof of Proposition~4.3,
the linear transformation $u_i\mapsto d_i^{-1}u_i$, $i=1,\dots,k$,
identifies the faces $\sigma$ of $\Delta(f)$ with the faces $\sigma'$ of
$\Delta(g)$.  In particular, the face $\Delta(f_A)$ of $\Delta(f)$
will correspond to some face $\sigma'_A$ of $\Delta(g)$.  Let $g_A$
denote the sum of those terms of $g$ whose exponents lie on the face
$\sigma'_A$ (so that $\Delta(g_A) = \sigma'_A$).  By the Remark
preceding Proposition~4.3, we have 
\begin{equation}
L({\mathbb T}^{n-|A|},g^{(A)};t)^{(-1)^{n-|A|-1}} = L({\mathbb
  T}^{n-|A|},f_A;t)^{(-1)^{n-|A|-1}} =
P_A(t)^{\delta^{n-\dim\Delta(f)}}. 
\end{equation}
The nondegeneracy of $f_A$ relative to $(\Delta(f_A),M_{J_A})$ implies
the nondegeneracy of $g_A$ relative to $(\Delta(g_A),{\mathbb
  Z}^{\dim\Delta(g_A)})$.  We can thus apply the results of \cite{AS4}
and \cite{DL} to $g_A$ to compute the number of roots of $P_A(t)$ of a
given archimedian weight.  By (4.22) and the fact that $Q(t)$ is a
polynomial, we then get the number of roots of $Q(t)$ of a given
archimedian weight.

For applications in the next section, we calculate the number of
reciprocal roots of largest possible archimedian absolute value
$q^{(\dim\Delta(f))/2}$ of $Q(t)$.  For $A\neq\emptyset$, all
reciprocal roots of $P_A(t)$ have absolute value
$<q^{(\dim\Delta(f))/2}$, so this is just the number of reciprocal
roots of $P_\emptyset(t)$ of absolute value
$q^{(\dim\Delta(f))/2}$. By (4.30), this can be obtained by
applying \cite[Theorem~1.10]{AS4} to $g$:  the number
$w_{\dim\Delta(f)}$ of reciprocal roots of highest weight is
\begin{equation}
w_{\dim\Delta(f)} =
\sum_{(0,\dots,0)\subseteq\sigma'\subseteq\Delta(g)}
(-1)^{\dim\Delta(g)-\dim\sigma'} (\dim\sigma')!\,V_{{\mathbb
    Z}^{\dim\sigma'}}(\sigma'). 
\end{equation}
Since $\Delta(g)$ is obtained from $\Delta(f)$ by an explicit linear
transformation, we can express this in terms of invariants of
$\Delta(f)$:
\begin{equation}
w_{\dim\Delta(f)} =
\sum_{(0,\dots,0)\subseteq\sigma\subseteq\Delta(f)}
(-1)^{\dim\Delta(f)-\dim\sigma} (\dim\sigma)!\,V_{M_{J_\sigma}}(\sigma),
\end{equation}
where $J_\sigma = J\cap\sigma$.

We note an important special case of this formula.  If every face of
$\Delta(f)$ that contains the origin is of the form $\Delta(f_A)$ for some
$A\subseteq\{k+1,\dots,n\}$, the right-hand side of (4.32) is just
$\nu(f)$.  This gives the following result.

\begin{corollary}
Under the hypothesis of Theorem $4.20$, if every face of $\Delta(f)$
that contains the origin is of the form $\Delta(f_A)$ for some
$A\subseteq\{k+1,\dots,n\}$, then all reciprocal roots of $Q(t)$ have
archimedian absolute value $q^{(\dim\Delta(f))/2}$.
\end{corollary}

As a special case of Corollary 4.33, we note the following result.
\begin{corollary}
If $f\in{\mathbb F}_q[x_1,\dots,x_n]$ is nondegenerate relative to
$(\Delta(f),M_J)$ and convenient, then $L({\mathbb
  A}^n,f;t)^{(-1)^{n-1}}$ is a polynomial of degree $\nu(f)$ all of
whose reciprocal roots have absolute value $q^{n/2}$. 
\end{corollary}

\section{Examples}

We explain how Theorem 4.20 implies a generalization of the
result of N. Katz quoted in the Introduction.  
\begin{proposition}
Let $f\in{\mathbb F}_q[x_1,\dots,x_n]$ have degree
$d=p^ke$, $(e,p) = 1$, and 
suppose that every monomial appearing in $f$ has degree divisible by $p^k$.
If $f^{(d)}$, the homogeneous part of $f$ of degree $d$, defines a smooth
hypersurface in ${\mathbb P}^{n-1}$, then $L({\mathbb A}^n,f;t)^{(-1)^{n-1}}$
is a polynomial of degree
\begin{equation}
\nu(f) = \frac{1}{p^k}((d-1)^n + (-1)^n(p^k-1))
\end{equation}
all of whose reciprocal roots have absolute value $q^{n/2}$.
\end{proposition}

\begin{proof}
Let ${\bf e}^{(1)},\dots,{\bf e}^{(n)}$ denote the standard basis for
${\mathbb R}^n$.  Over any sufficiently large extension field of ${\mathbb
  F}_q$, we can make a coordinate change on ${\mathbb A}^n$ so that $f$ is
convenient and for any $A\subseteq\{1,\dots,n\}$ the intersection of $f^{(d)}
= 0$ with the coordinate hyperplanes $\{x_i = 0\}_{i\in A}$ is smooth.  In
particular, the equations $f_A^{(d)} = 0$ define smooth hypersurfaces in
${\mathbb P}^{n-|A|-1}$.  The Newton polyhedron $\Delta(f)$ is then the
simplex in ${\mathbb R}^n$ with vertices at the origin and the points
$\{d{\bf e}^{(i)}\}_{i=1}^n$.  The faces of $\Delta(f)$ not containing the
origin are the convex hulls of the sets $\{d{\bf e}^{(i)}\}_{i\in A}$.  It
will be simpler to index these faces by their complements: let $\sigma_A$
denote the face which is the convex hull of $\{d{\bf e}^{(i)}\}_{i\not\in A}$. 

Write $f=\sum_{j\in J}a_jx^j$, $J$ a finite subset of ${\mathbb N}^n$.  Let
$M\subseteq{\mathbb Z}^n$ be the subgroup
\[ M = \{(u_1,\dots,u_n)\in{\mathbb Z}^n\mid \text{$u_1+\dots+u_n$ is
  divisible by $p^k$}\}. \]
Since all monomials in $f$ have degree divisible by $p^k$, it follows
that ${\mathbb Z}\langle J\rangle\subseteq M$.  In fact, $M_J\subseteq M$.  To
see this, let $(u_1,\dots,u_n)\in M_J$.  By definition, there exists an
integer $c$ prime to $p$ such that $c(u_1,\dots,u_n)\in{\mathbb Z}\langle
J\rangle$.  This implies that $c\sum_{i=1}^n u_i$ is divisible by $p^k$.  But
since $(c,p) = 1$, one has $\sum_{i=1}^n u_i$ divisible by $p^k$, therefore
$(u_1,\dots,u_n)\in M$. 

We claim that $f$ is nondegenerate relative to $(\Delta(f),M)$.  As basis for
$M$ we take the elements
\[ (p^k,0,\dots,0)\cup\{(-1,0,\dots,0,1,0,\dots,0)\}_{i=2}^n, \]
where the ``1'' occurs in the $i$-th position, and as basis for $L={\rm
  Hom}_{\mathbb Z}(M,{\mathbb Z})$ we take the ``dual basis'', namely, the
  linear forms  
\[ \ell_1(u_1,\dots,u_n) = p^{-k}(u_1+\dots+u_n) \]
and
\[ \ell_i(u_1,\dots,u_n) = u_i \]
for $i=2,\dots,n$.  Let $A\subseteq\{1,\dots,n\}$ and let $\sigma_A$ be the
face of $\Delta(f)$ defined above.  Note that
\[ f_{\sigma_A} := \sum_{j\in J\cap\sigma_A} a_jx^j = f^{(d)}_A. \]
We must thus check that $\{E_{\ell_i}(f^{(d)}_A)\}_{i=1}^n$ have no common
zero in $(\bar{\mathbb F}_q^\times)^n$.  But 
\[ E_{\ell_1}(f^{(d)}_A) = e^{-1}f^{(d)}_A \] 
and
\[ E_{\ell_i}(f^{(d)}_A) = x_i\frac{\partial f^{(d)}_A}{\partial x_i} \]
for $i=2,\dots,n$, so we must show that the system 
\begin{equation}
f^{(d)}_A = x_2\frac{\partial f^{(d)}_A}{\partial x_2} = \dots =
  x_n\frac{\partial f^{(d)}_A}{\partial x_n} = 0 
\end{equation}
has no solution in $(\bar{\mathbb F}_q^\times)^n$.  Since $p\,|\, d$,
the Euler relation implies that any common zero of $\{x_i\partial
f_A^{(d)}/\partial x_i\}_{i=2}^n$ is also a zero of $x_1\partial
f_A^{(d)}/\partial x_1$, thus the system $(5.3)$ is equivalent to the system
\begin{equation}
f^{(d)}_A = x_1\frac{\partial f^{(d)}_A}{\partial x_1} = \dots =
  x_n\frac{\partial f^{(d)}_A}{\partial x_n} = 0.
\end{equation}
Furthermore, $x_i$ does not appear in $f_A$ if $i\in A$, hence the solutions
of (5.4) in $(\bar{\mathbb F}_q^\times)^n$ are exactly the solutions of the
set 
\begin{equation}
\{f_A^{(d)}\}\cup \{\partial f^{(d)}_A/\partial x_i\}_{i\not\in A} 
\end{equation}
in $(\bar{\mathbb F}_q^\times)^n$.  However, the equation $f^{(d)}_A=0$
defines a smooth hypersurface in ${\mathbb P}^{n-|A|-1}$, so any common zero
of the set (5.5) must have $x_i=0$ for all $i\not\in A$.  In particular,
(5.5) has no common zero in $(\bar{\mathbb F}_q^\times)^n$.  This implies
that (5.4) has no solution in $(\bar{\mathbb F}_q^\times)^n$, proving the
nondegeneracy of $f$ relative to $(\Delta(f),M)$.  

We can now compute $\nu(f)$.  By Proposition 4.5(a) we have $M=M_J$, so
\[ [{\mathbb Z}^{n-|A|}:M_{J_A}] = p^k\quad \text{for all
  $A\neq\{1,\dots,n\}$} \]
and
\[ (n-|A|)!\,V(f_A)/[{\mathbb Z}^{n-|A|}:M_{J_A}] = \begin{cases}
  d^{n-|A|}/p^k & \text{if $A\neq\{1,\dots,n\}$,} \\ 1 & \text{if $A
  = \{1,\dots,n\}$.} \end{cases} \]
Then clearly
\[ \nu(f) = \frac{1}{p^k}((d-1)^n + (-1)^n(p^k-1)) \]
and the assertions of Proposition~5.1 follow from Theorem~4.20.  Finally, note
that if $L({\mathbb A}^n,f;t)^{(-1)^{n-1}}$ is a polynomial of degree (5.2)
over all sufficiently large extension fields of ${\mathbb F}_q$, then the
same is true over ${\mathbb F}_q$ itself.  The assertion about the
absolute value of the roots follows immediately from Corollary~4.34.
\end{proof}

{\it Remark}.  There are many results in the literature that, like
Proposition~5.1, assert that $L({\mathbb A}^n,f;t)^{(-1)^{n-1}}$ is a
polynomial if $f^{(d)}$ defines a smooth hypersurface and some
additional condition is satisfied (see \cite[Th\'{e}or\`{e}me 8.4]{D},
\cite[Theorem~1.11 and the following remark]{AS5},
\cite[Theorem~3.6.5]{K}, \cite[Theorem~3.1]{AS6}).  One might ask if
any additional condition is really necessary.  Consider the
three-variable polynomial 
\[ f = (z^p-z) +x^{p-1}y + y^{p-1}z. \]
The homogeneous part of degree $p$ is smooth but $f$ has the same
$L$-function as
\[ g = x^{p-1}y + y^{p-1}z. \]
Since $\sum_{z\in{\mathbb F}_q} \Psi(y^{p-1}z) = 0$ if $y\neq 0$, one
calculates that $\sum_{x,y,z\in{\mathbb F}_q}\Psi(g(x,y,z)) = q^2$.
This gives $L({\mathbb A}^3,f;t) = (1-q^2t)^{-1}$, showing that
smoothness of $f^{(d)}$ alone is not sufficient to guarantee that
$L({\mathbb A}^n,f;t)^{(-1)^{n-1}}$ will be a polynomial.

We apply Theorem 4.20 to compute the zeta functions of some possibly
singular hypersurfaces.  Let $f\in{\mathbb F}_q[x_1,\dots,x_n]$ be a
homogeneous polynomial and let $X\subseteq{\mathbb P}^{n-1}$ be the
hypersurface $f=0$.  Write the zeta function $Z(X/{\mathbb
  F}_q,t)$ of $X$ in the form 
\begin{equation}
Z(X/{\mathbb F}_q,t) =
  \frac{R(t)^{(-1)^{n-1}}}{(1-t)(1-qt)\dots (1-q^{n-2}t)}, 
\end{equation}
where $R(t)$ is a rational function.  The exponential sum associated
to the polynomial $yf\in{\mathbb F}_q[x_1,\dots,x_n,y^{\pm 1}]$ can be
used to count points on the projective hypersurface~$X$.  The precise
relation is given in \cite[Eq.~(6.14)]{AS3}:
\begin{equation}
L({\mathbb A}^n\times{\mathbb T},yf;t)^{(-1)^n} =
R(qt)^\delta.  
\end{equation}
\begin{proposition}
Suppose that $yf\in{\mathbb F}_q[x_1,\dots,x_n,y^{\pm 1}]$ is
nondegenerate relative to $(\Delta(yf),M_J)$ and convenient.  Then
$R(t)$ is a polynomial of degree $\nu(yf)$ all of whose reciprocal
roots have absolute value $q^{(n-2)/2}$.
\end{proposition}

\begin{proof}
The assertion about the degree of $R(t)$ follows immediately by
applying Theorem~4.20 to~(5.7).  The assertion about the absolute
values of the roots of $R(t)$ follows immediately from Corollary~4.33.
\end{proof}
 
As an illustration of Proposition 5.8, consider the projective
hypersurface $X\subseteq{\mathbb P}^{n-1}$ over 
${\mathbb F}_q$ defined by the homogeneous equation 
\[ f(x_1,\dots,x_n) = x_1^n + \dots +x_n^n + \lambda x_1\dots x_n =
  0, \] 
where $\lambda\in{\mathbb F}_q$.  If $p\nmid n$, this hypersurface is
smooth for all but finitely many values of $\lambda$.  If $p\,|\,n$,
it is a singular hypersurface for all nonzero $\lambda$ (except in the
cases $p=n=2$ and $p=n=3$).  We describe the zeta function when
$p\,|\,n$. 

\begin{corollary}
Suppose that $n=p^ke$, where $k\geq 1$ and $(p,e) = 1$, and $\lambda\neq 0$.
Then $R(t)$ is a polynomial of degree
\begin{equation}
\deg R(t) = (p^k-1)e^{n-1} + e^{-1}((e-1)^n + (-1)^n(e-1))
\end{equation}
all of whose reciprocal roots have absolute value $q^{(n-2)/2}$.
\end{corollary}

{\it Remark}.  Note that the second summand on the right-hand side of
(5.10) is the dimension of the primitive part of middle-dimensional
cohomology of a smooth hypersurface of degree $e$.  When $\lambda =
0$, the hypersurface $X_0$ is smooth of degree $e$. (It is defined by
the equation $x_1^e + \dots + x_n^e = 0$.)

\begin{proof}[Proof of Corollary $5.9$]
The proof is a direct application of Proposition 5.8.  We sketch the
details.  It is straightforward to check that $yf$ is convenient:
$\Delta(yf)$ is the $n$-simplex in ${\mathbb R}^{n+1}$ with vertices
at the origin and the points 
\[ (n,0,\dots,0,1),(0,n,0,\dots,0,1),\dots,(0,\dots,0,n,1) \]
and for each subset $A\subseteq\{1,\dots,n\}$, one has $\dim\Delta(yf_A) =
n-|A|$.  We have
\[ J = \{(n,0,\dots,0,1),(0,n,0,\dots,0,1),\dots,(0,\dots,0,n,1),
(1,\dots,1,1)\}\subseteq{\mathbb Z}^{n+1}, \]
thus ${\mathbb R}\langle J\rangle$ is the hyperplane in ${\mathbb R}^{n+1}$
with equation $u_1 + \dots + u_n = nv$ and the lattice ${\mathbb
  Z}^{n+1}\cap{\mathbb R}\langle J\rangle$ has basis  
\[ B = \{(1,-1,0,\dots,0),(0,1,-1,0,\dots,0),\dots,(0,\dots,0,1,-1,0),
(0,\dots,0,n,1)\}. \]
It follows that $n!\,V_n(yf) = n^{n-1}$.  Similarly, we have
\[ (n-|A|)!\,V_{n-|A|}(yf_A) = \begin{cases} n^{n-1-|A|} & \text{if $|A|\leq
    n-1$,} \\ 1 & \text{if $|A| = n$.} \end{cases} \]

Let the first $n-1$ vectors in $B$ be denoted ${\bf a}_i$, $i=1,\dots,n-1$.
The lattice ${\mathbb Z}\langle J\rangle$ has basis 
\[ n{\bf a}_1,\dots,n{\bf a}_{n-2},(n-1,-1,\dots,-1,0), (1,\dots,1,1), \]
from which it follows that $M_J$ has basis
\begin{equation}
p^k{\bf a}_1,\dots,p^k{\bf a}_{n-2},(n-1,-1,\dots,-1,0), (1,\dots,1,1). 
\end{equation}
One then checks that
\[ [{\mathbb Z}^{n+1}\cap{\mathbb R}\langle J\rangle:M_J] = (p^k)^{n-2}. \]
For $|A|\geq 1$, the calculation is easier as $J_A$ consists of vectors
$(0,\dots,0,n,0,\dots,0,1)$ for which the ``$n$'' occurs in the $i$-th entry
for $i\not\in A$ (the vector $(1,\dots,1,1)$ does not appear).  One gets
\[ [{\mathbb Z}^{n+1}\cap{\mathbb R}\langle J_A\rangle:M_{J_A}] =
\begin{cases} (p^k)^{n-2} & \text{if $A=\emptyset$,} \\ (p^k)^{n-1-|A|} &
  \text{if $1\leq |A|\leq n-1$,} \\ 1 & \text{if $A = \{1,\dots,n\}$.}
\end{cases} \] 
We then have
\[ \frac{(n-|A|)!\,V_{n-|A|}(yf_A)}{[{\mathbb Z}^{n+1}\cap{\mathbb R}\langle
J_A\rangle:M_{J_A}]} = \begin{cases} p^ke^{n-1} & \text{if $A = \emptyset$,}
  \\ e^{n-1-|A|} & \text{if $1\leq |A|\leq n-1$,} \\ 1 & \text{if $A =
  \{1,\dots,n\}$.} \end{cases} \]
It is now straightforward to check that $\nu(yf)$ equals the
expression on the right-hand side of~(5.10).

It remains to check that $yf$ is nondegenerate relative to
$(\Delta(yf),M_J)$.  The dual basis of the basis (5.11) for $M_J$ is
the set of linear forms 
\begin{align*}
\ell_i(u_1,\dots,u_n,v) &= \sum_{j=1}^i \frac{1}{p^k}u_j + \frac{n-i}{p^k}
u_n - ev, \quad\text{$i=1,\dots,n-2$}, \\
\ell_{n-1}(u_1,\dots,u_n,v) &= -u_n + v, \\
\ell_n(u_1,\dots,u_n,v) &= v. 
\end{align*}
The polynomials $(yf)_\sigma$ for faces $\sigma$ of $\Delta(yf)$ that
do not contain the origin are exactly the polynomials $yf_A$ for
$A\subset\{1,\dots,n\}$, $|A|<n$.  If $A=\emptyset$, we have
\[ E_{\ell_n}(yf)-E_{\ell_{n-1}}(yf) = \lambda yx_1\dots x_n, \] 
which has no zero in $(\bar{\mathbb F}_q^\times)^{n+1}$.  So suppose that
$1\leq |A|\leq n-1$.  Then
\[ yf_A = \sum_{i\not\in A} yx_i^n. \]

Suppose first that $n\not\in A$.  If $1\in A$, then
\[ E_{\ell_1}(yf_A)+eE_{\ell_n}(yf_A) = -eyx_n^n, \]
and if $i\in A$ for some $i$, $2\leq i\leq n-2$, then 
\[ E_{\ell_i}(yf_A)-E_{\ell_{i-1}}(yf_A) = -eyx_n^n. \]
Neither of these monomials vanishes on $(\bar{\mathbb F}_q^\times)^{n+1}$. 
If $i\not\in A$ for all $i=1,\dots,n-2$, then $A = \{n-1\}$.  In this
case we have  
\[ E_{\ell_1}(yf_A)+eE_{\ell_n}(yf_A) = ey(x_1^n-x_n^n),
\] 
\[ E_{\ell_i}(yf_A)-E_{\ell_{i-1}}(yf_A) = ey(x_i^n-x_n^n) \] 
for $i=2,\dots,n-2$, and
\[ E_{\ell_n}(yf_A) = y(x_1^n + \dots + x_{n-2}^n
+ x_n^n). \]
If the first $n-2$ expressions vanish, then $yx_1^n = \dots = yx_{n-2}^n =
yx_n^n$.  The vanishing of the last expression is then equivalent to
$(n-1)yx_n^n = 0$, which is impossible in~$(\bar{\mathbb F}_q^\times)^{n+1}$. 

Now suppose that $n\in A$.  If $1\not\in A$, then
\[ E_{\ell_1}(yf_A)+eE_{\ell_n}(yf_A) = eyx_1^n \]
and if $i\not\in A$ for some $i$, $2\leq i\leq n-2$, then 
\[ E_{\ell_i}(yf_A)-E_{\ell_{i-1}}(yf_A) = eyx_i^n. \]
Neither of these monomials vanishes on $(\bar{\mathbb F}_q^\times)^{n+1}$. 
If $i\in A$ for $i=1,\dots,n-2$, then $A$ contains all indices except $i=n-1$
and $E_{\ell_n}(yf_A) = yx_{n-1}^n$, which does not vanish on
$(\bar{\mathbb F}_q^\times)^{n+1}$.   

Thus $yf$ satisfies the hypotheses of Proposition 5.8. 
\end{proof}

\end{document}